\documentclass[12pt]{amsart}
\usepackage{amsmath, amsthm, amscd, amsfonts, amssymb, graphicx, color,mathabx}
\usepackage[margin=2.8cm]{geometry}
\usepackage[bookmarksnumbered, colorlinks, plainpages]{hyperref}

\newtheorem{theorem}{Theorem}[section]
\newtheorem{lemma}[theorem]{Lemma}

\newtheorem{corollary}[theorem]{Corollary}
\theoremstyle{definition}

\newtheorem{example}[theorem]{Example}

\newtheorem{question}[theorem]{Question}

\theoremstyle{remark}

\numberwithin{equation}{section}

\newcommand{\N}{\mbox{$\mathbb{N}$}}

\DeclareMathOperator{\End}{End}
\DeclareMathOperator{\supp}{supp}
\DeclareMathOperator{\Ann}{Ann}

\allowdisplaybreaks
\begin{document}

\title{On Reduced archimedean skew power series rings}

%\title{The ascending chain conditions for principal ideals of skew generalized power series rings and }

\author{Hamed Mousavi, Farzad Padashnik and Ayesha Asloob Qureshi$*$}
\address{Farzad Padashnik, Department of Pure Mathematics, Faculty of Mathematical
Sciences, Isfahan University, Isfahan, Iran.}
\email{f.padashnik1368@gmail.com}
\address{Hamed Mousavi, Mathematics Department, Georgia Institute of Technology,  686 Cherry Street NW, Atlanta, Georgia 30313, US.}
\email{hmousavi6@gatech.edu}
\address{Ayesha Asloob Qureshi, Sabanc\i \; University, Faculty of Engineering and Natural Sciences, Orta Mahalle, Tuzla 34956, Istanbul, Turkey}
\email{aqureshi@sabanciuniv.edu}
\date{}

\begin{abstract}
In this paper, we prove that if $R$ is an Archimedean reduced ring and satisfy ACC on annihilators, then $R[[x]]$ is also an Archimedean reduced ring.
% In fact, we prove a quite general result in this direction. Let $(S,\leq)$ be an Archimedean, positive, strictly totally ordered monoid and $R$ be an $S$-rigid ring satisfying the \emph{ACC} on annihilators, $\omega: S\rightarrow \End(R)$ a monoid homomorphism and $\omega_s$
%be surjective for each $s\in S$. We show that $R$ is right Archimedean if and only if the skew generalized power series ring $R[[S,\omega]]$ is a right Archimedean reduced ring. In particular, we deduce
More generally we prove that if $R$ is a right Archimedean ring  satisfying the \emph{ACC} on annihilators and $\alpha$ is a rigid automorphism of $R$, then the skew power series ring $R[[x;\alpha ]]$ is right Archimedean reduced ring. 
% In fact, we mention that our proof is working in quite general case of Skew generalized power series rings in this direction.
We also provide some examples to justify the assumptions we made to obtain the required result.
\end{abstract}

\keywords{Skew power series ring, generalized power series rings, right (or left) Archimedean ring, Annihilator, Ascending chain conditions for principal one-sided ideals, ordered monoids}

\subjclass[2010]{16P70, 16P60, 13F10, 13J05}
\maketitle
\section{Introduction}
A ring $R$ is called \textit{left Archimedean}, if for each nonunit element $r\in R$ we have $\bigcap_{n\in \mathbb{N}}r^nR=\{0\}$. The right Archimedean rings are defined in a similar way. In \cite{Sheldon}, it is shown that if $R$ is a domain and $\bigcap_{n\in \mathbb{N}} a^nR=0$, then the quotient field $Q(R[[x/a]])$ has infinite transcendent degree over the quotient field $Q(R[[x]])$. A very important result in this direction is  that if $R$ is a domain (or is completely integrally closed) and satisfies \emph{ACCP}, then $R$ is Archimedean, see \cite{andersonanti,Sheldon}. However, each \emph{ACCP}-domain is not necessarily Archimedean, see for example, \cite[p. 1127]{Dumitrescu}. There are similar results  in \cite{ArnoldII,condo}.   

%In fact, if $R$ is a domain, then by T. Dumitrescu et al. \cite[Remark 1.1]{Dumitrescu} and Z. Liu \cite[Lemma 3.1]{liu}, a domain $R$ satisfies \emph{ACCPL} if and only if $\bigcap_{n\in \mathbb{N}} r_1r_2 \cdots r_nR = \{0\}$ for any sequence $(r_n)_{n\in \mathbb{N} }$ of nonunits of $R$. 

Ribenboim in \cite{Ribenboimsemisimple} defined the ring of
generalized power series $R[[S]]$ consisting of all maps from $S$ to
$R$ whose support is Artinian and narrow  with  the pointwise
addition and the convolution multiplication. This construction
provided interesting examples of rings, for example see \cite{Ribenboim (1995a),Ribenboim (1995b)}. In \cite{Ribenboim}, Ribenboim gave a sufficient condition for the ring $R[[S]]$ (not necessarily commutative) to be Noetherian.  However, Frohn  gave an example to show that \emph{ACCP} does not rise to the power series ring in general (see \cite{counter}). 

In \cite{zim}, R. Mazurek and M. Ziembowski, introduced a ``twisted" version of the Ribenboim
construction and studied when it produces a von Neumann regular ring.   
%In \cite{zim}, R. Mazurek and M. Zimbowski 
They also proved  that if $R$ is a domain and $\omega$ an endomorphism of  $R$, then  $R$ satisfies \emph{ACCPL}, $S$ is an \emph{ACCPL}-monoid and $\omega_s$  is injective  for each $s\in S$ if and only if $R[[S,\omega]]$ is an \emph{ACCPL}-domain. 

As a similar result, for a strictly ordered monoid $S$ and a monoid homomorphism $\omega : S \rightarrow \End (R)$, in \cite{paykan} it is proved that if the attached skew generalized power series ring $R[[S,\omega]]$ is a right Archimedean domain then $R$ is a right Archimedean domain, $S$ is a right Artinian and narrow monoid and $\omega_s$ is an injective which preserves nonunits of $R$ for any $s \in S$.
Next step in this direction, 
%we present our main theorem in section \ref{archh}. In Theorem~\ref{rigidar}, 
we show that when $R$ is right (resp. left) Archimedean and satisfies the \emph{ACC} on  annihilators, then $R[[x;\alpha]]$ is a right (resp. left) Archimedean ring. 
%This proof can leverage a way to prove the same result for the generalized skew power series $R[[S,\omega]]$, which we skip the proof due to simplicity.
In section \ref{archh}, we give an example to justify that the condition given in Theorem \ref{rigidar} is necessary on the ring $R$, to be able to pass the property of being Archimedean between $R$ and $R[[x;\alpha]]$.

A ring $R$ is \textit{reduced}, if the condition $a^2=0$ results in $a=0$. According to Krempa \cite{krem}, an endomorphism $\alpha$ of a ring $R$ is said to be \textit{rigid} if $a\alpha(a)=0$ implies $a=0$, for $a \in R$. A ring $R$ is said to be $\alpha$-\textit{rigid} if there exists a rigid endomorphism $\alpha$ of $R$. Clearly, every domain $D$ with a monomorphism $\alpha$ is rigid. It is clear that, $\alpha$-rigid rings are reduced, (see e.g., \cite{ore}).
% Given a monoid $S$, the ring $R$ is called $S-$rigid if $a\omega_s(a)=0$ results in $a=0$ for all $s\in S$.
The following Lemma  by Hong in \cite{ore} is very useful  in the following sections.

\begin{lemma}\label{jav}
Let $R$ be a $\alpha$-rigid ring and $a,b \in R$. Then 
\begin{itemize}
\item[(1)] $R$ is reduced.
\item[(2)] If $ab=0$, then  we have
$a\alpha^{n}(b)=\alpha^{n}(a)b=0$  for any positive integer $n$.

\end{itemize}
\end{lemma}
\begin{proof}
(1) Let $a^2=0$. Then $a\alpha(a)\alpha(a)\alpha^2(a)=0$. So $a\alpha(a)\alpha(a\alpha(a))=0$. Because $R$ is rigid $a\alpha(a)=0$. Another use of it gives us $a=0$.

(2) Let $ab=0$. So $ba=0$. This implies that $a\alpha^n(ba)\alpha^{2n}(b)=0$. So $a\alpha^n(b)\alpha^n(a\alpha^n(b))=0$. Since $R$ is $\alpha-$rigid, then $a\alpha^n(b)=0$. Similarly $\alpha^n(a)b=0$.
\end{proof}

%Finally, we give a very useful result on Archimedean monoid. 
%%According to \cite[Page 922]{brookfield}, if $S$ be a positively ordered monoid and $a,b\in S$, then $a<b$ implies that there exists $c\in S$ such that $ac=b$.
%
%\begin{lemma}\label{arclem}
%Let $S$ be a positive strictly totally ordered Archimedean monoid. Then $s\geq 1_S$ for every $s\in S$.  
%\end{lemma}
%\begin{proof}
%Let $s<1_S$. As $S$ is positively ordered monoid, assume that $s's=1_S$. So we can say that $s's^k=s^{k-1}$. This implies that $s'^{n-1}s^{n}=s$. This implies that $s\in \cap_{n=1}^{\infty} Ss^n$. As $S$ is an Archimedean monoid, $s$ should be equal to $1_S$ which is impossible.   
%\end{proof}
\medskip
Throughout this paper, all monoids and rings are with identity element that is inherited by submonoids and subrings and preserved under homomorphisms, but neither monoids nor rings are assumed to be - necessarily -  commutative.\par

\section{Archimedean skew power series rings}\label{archh}
In this section, we give the conditions on the ring $R$ such that the property of being Archimedean can be transferred from $R$ to the skew power series ring $ R[[x;\alpha]]$ and vice versa. The following results in \cite[Theorems 4.1]{paykan} describe such a transfer with the assumption that $ R[[S,\omega]]$ is a domain. 
% and to do this we first set recall some notation and definition.%The detailed proof is in the submitted paper \cite{paykan}. 

Let  $\omega: S\longrightarrow \End(R)$ be a monoid homomorphism, where $R$ is a ring and $(S,\le)$ is a strictly ordered monoid. If $R[[S,\omega]]$ is a left (resp. right) Archimedean domain, then one can see that $R$ is a left (resp. right) Archimedean domain. It can be proved straightforward from the definition and the fact that $\bigcap _{n\in \mathbb{N}}(c_a)^nR[[S,\omega]]= 0$ for nonunit $a$. Also, $S$ is a left (resp. right) Artinian and  narrow monoid because of the fact that $\bigcap _{n\in \mathbb{N}}(\textbf{e}_s)^nR[[S,\omega]]=0$ for nonunit $s\in S$. Moreover, $\omega _s$ is injective (and preserve nonunits of $R$) for any $s\in S$. 
%Indeed, if $\omega_s(r)=0$, then $\textbf{e}_{s_0}=f_n(c_r)^n$ where $f_n=c_{(\omega_{s_0}(r))^{-n}}\textbf{e}_{s_0}$. This implies that $\textbf{e}_{s_0}=0$ or $r$ is a unit. But $r\in ker(\omega_s)$, So none of these cases is possible.

As mentioned in previous section, if $(S,.,\leq)$ is a (in particular positive) strictly totally ordered monoid and $0\neq f \in R[[S,\omega]]$, then $\supp( f )$ is a nonempty well-ordered subset of $S$. The smallest element of $\supp(f)$ is denoted by $\pi( f )$ and we set $\theta(f):=f(\pi(f))$. In particular, for a power series $f(x)=\sum f_nx^n$, with the first nonzero coefficient is $f_m$, we have $\pi(f)=m$ and $\theta(f)=f_m$.
 Let $U(R)$ be the set of unit elements of $R$. Then it is shown  in \cite[Proposition 3.2]{zim} that if $\pi(f) \in U(S)$ and $\theta (f) \in U(R)$, then $f \in U(R[[S,\omega]])$, where $f\in R[[S,\omega]] \setminus \{0\}$. 
%\textcolor{blue}{There is also the following theorem in \cite{paykan}.} 
\begin{theorem}\cite[Theorems 4.1]{paykan}\label{paykan}
Let $R$ be a ring, $(S,\le)$ a strictly ordered monoid and $\omega: S\longrightarrow \End(R)$ a monoid homomorphism. Then we have the following:
\begin{itemize}
\item[{\em (i)}] The ring $R[[S,\omega]]$ is a right Archimedean domain if and only if $R$ is a
right Archimedean domain, $\omega _s$ is injective and preserves nonunits of $R$ for any $s\in S$. 

\item[{\em (ii)}] The ring $R[[S,\omega]]$ is a left Archimedean domain if and only if $R$ is a left Archimedean domain and $\omega _s$ is injective for any $s\in S$. 
\end{itemize}
\end{theorem}

%\textbf{Sketch of the proof:} The implication that if $R[[S,\omega]]$ is Archimidean domain then $R$ is Archimidean domain is clear. Conversely, assume that $R$ is a right Archimedean domain, $\omega _s$ is injective and preserves nonunits of $R$ for any $s\in S$. If $A=R[[S,\omega]]$ is not a right Archimedean domain then there exists and element $f$ in $A$ such that $g \in \bigcap _{n\ge 1}Af^n$ and $g \neq 0$. Then for each $n\in \mathbb{N}$ there exists $h_n \in A$ such that $g=h_nf^n$. One can see that $\theta(g)=\theta(h_n) \omega_{\pi(h_n)} (\theta(f^n))$. If $\pi(f)=1$, then $f$ should be unit in $R[[S,\omega]]$. Otherwise 
%\begin{align}
%\theta (g)=h_n(\pi (g)) (\omega_{\pi (g)}(f(1)))^n \in  R(\omega_{\pi (g)} (f(1)))^n
%\end{align}
%for every $n$. This and the fact that $R$ is Archimedean implies that $\theta(g)=0$ which is not possible. So $\pi(f)>1$ or $\pi(f)<1$. If $\pi (f)>1$, then $\pi (h_n)<\pi(h_{n-1})$ for each $n$. This contradicts with the fact that $S$ is positively Archimedean narrow. If $\pi(f)<1$, then $(\pi (f))^n< (\pi (f))^{n-1}$ for any $n$. This also contradicts to the same assumption for $S$. So the assumption $g\in \bigcap _{n\ge 1}Af^n$ results in $f$ is a unit. This is the definition of $R[[S,\omega]]$ is Archimidean.
%\qed

 Before stating the main result of the paper, we first recall a few definitions. An ideal $P$ of $R$ is \textit{completely prime} if $ab\in P$ implies $a\in P$ or $b\in P$ for $a,b\in R$.  Let $\alpha$ an endomorphism of $R$. An ideal $I$ of $R$ is called an \textit{$\alpha$-ideal} if $\alpha (I)\subseteq I$ and $I$ is called \textit{$\alpha$-invariant} if $\alpha ^{-1}(I)=I$. Note that if $I$ is an $(\alpha)$-ideal, then $\bar{\alpha}:R/I\longrightarrow R/I$ defined by $\bar{\alpha}(a + I)=\alpha (a)+I$ for $a\in R$ is an endomorphism of the factor ring $R/I$. 

%In the next theorem, we discuss the transfer of the property of being Archimedean with a weaker assumption than being a domain. 
We prove a similar result in the realm of reduced rings. Reduced rings are natural generalization of domains, for example direct product of domains is reduced. Therefore, Archimedean rings can also keep certain properties similar to the Archimedean domains.

\begin{theorem}\label{rigidar}
Let  $R$ be a $\alpha$-rigid ring satisfying the ACC on annihilators and 
%Furthermore, let
%$\omega: S \rightarrow \End(R)$
%be a monoid homomorphism and 
$\alpha$  be a surjective endomorphism.
%for all $s\in S$ 
Then:
\begin{enumerate}
\item[\em{(1)}] $R$ is right Archimedean and $\alpha$ preserves nonunits of $R$ if and only if the reduced ring $R[[x;\alpha]]$ is right Archimedean .
\item[\em{(2)}] $R$ is left Archimedean if and only if the reduced ring $R[[x;\alpha]]$ is left Archimedean.
\end{enumerate} 
 %\textcolor{red}{We can write this theorem like theorem \ref{domar}. Because for the right case, we need to preserve nonunits of $R$ for any $s\in S$.}
\end{theorem}

%\textcolor{red}{Hamed still, there are place in proof where $\omega_s$ and $\omega$ is used. In statement we have assumption for $\omega_s$ but then we are using $R[[S,\omega]]$ to prove it is right archimedean. I understand the difference and the use, but it doesnt look nice. In the above theorem of Paykan, they gave definiton of $\omega$ in the statement. Either you can add it like this or we can explicitly write it before the statement of the main theorem that what $\omega$ and $\omega_s$ stands for. I know we described it in the preliminary but it would be nice to have it right before the theorem so that one can use both terms wherever needed. }

%\textcolor{blue}{How about this Note that $\omega$ is a morphism (in category sense) from a monoid to the endomorphisms of a ring. So for $s\in S$, $\omega(s):=\omega_s$ is an endomorphism of $R$. So it is meaningless to say $\omega(r)$ for $r\in R$, we should say $\omega_s(r)$. }
%\begin{remark}
%\textcolor{blue}{The condition preserving nonunits in the right side case is necessary. Because we are citing Theorem \ref{domar} and we also use this property at step 4.}
%\end{remark}
\begin{proof}
We will give the proof for the right-sided version. The proof for the left-sided version follows on the same line if the order of the multiplication is changed from right to left throughout the proof. We recall that for a power series $g(x)=g_0+g_1x+\cdots $ we have $\pi(g):=m$ where $g_m\neq 0$  and $g_i=0$ for $0\leq i<m$. Also $\theta(g)=g_m$ is the leading term of the power series. We set $\pi(0)=-\infty$.

\medskip
 First, assume that $R$ is right Archimedean.  According to \cite{moussavi}, if $R$ is rigid, then $R[[x;\alpha]]$ is a reduced ring.  Let $A=R[[x;\alpha]]$ and 
 \[
 \Gamma := \{g \in A\; | \;  \text{ $g$ is a nonunit in $A$ and}\; \bigcap_{n\in \mathbb{N}}Ag^n\neq \{0\}\}.
 \]
We need to show that $\Gamma=\emptyset$. Assume that $\Gamma \neq \emptyset$.
%Assume on the contrary that $A=R[[S,\omega]]$ is not a right Archimedean ring. Then there exists a nonunit $g\in A$ such that $\bigcap_{n\in \mathbb{N}}Ag^n\neq \{0\}$. Set $\Gamma$ as the set of all such elements in $A$. 
We define  the set of leading terms $I_f:=\{\theta (g)|g\in fA\}\cup \{0\}$. Then, clearly $I_f$ is a right ideal of $R$.
Let
\[
T:=\{\Ann\big(\bigcup_{i\in \mathbb{N}}\;I_{c_ig}\big)|
%\Ann\big(\bigcup_{i\in \mathbb{N}}I_{b_ig}\big)\neq 0,
 \;  c_i \in R,\; g\in \Gamma \; \text{and}\;  \; c_ig\neq 0,\;c_ig^i=c_jg^j \; \forall \; i,j\in \mathbb{N}\}.
\]
%It is enough to prove that $T=\emptyset$. So assume that $T\neq\emptyset$.
 Then $T \neq \emptyset$ because we assumed that $\Gamma \neq \emptyset$. Note that $\Ann\big(I_{c_ig}\big)=\Ann\big(I_{c_ig^i}\big)$ because $R$ is rigid. Considering that $R$ satisfies \emph{ACC} on annihilators, and then applying Zorn's Lemma results in the fact that $T$ has a maximal element. Let $V:=\Ann\big(\bigcup_{i\in \mathbb{N}}I_{a_if}\big)$ be the maximal element of $T$ for some $f\in
\Gamma$ such that
\begin{align}\label{sharto}
 %a_1f=a_2f^2=\cdots =a_nf^n=\cdots , a_if\neq 0 \; \forall \; i\in \mathbb{N}.
 a_i \in R, \; a_if\neq 0,\; \text{and}\; a_if^i=a_jf^j \; \forall \; i,j\in \mathbb{N}.
\end{align}
%\textcolor{blue}{In the following arguments, we denote the coefficient $c_b$ by $b$.}% Again, we use $b$ instead of $c_b$.

\textbf{Step 1.} \textit{We claim that $V$ is a two-sided completely prime ideal of $A$. }
%In the following arguments, for simplicity, we denote the coefficient $c_b$ by $b$.

From Lemma \ref{jav}, we see that $R$ is reduced because it is rigid. It implies that annihilators in $R$ are two-sided ideals. Let $ab\in V$ for some $a,b\in R$ and $b\notin V$. We need to show that $a\in V$. %Then, $ab\theta (a_if)=0$ for all $i\in \mathbb{N}$. 
First, we show that $a\in \Ann(I_{ba_if})$ for all $i \in \N$. Observe that if $ba_ifh=0$ for all $h\in A$ and $i \in \N$, then $b\in V$ which contradicts our assumption that $b\notin V$.  Therefore, $ba_jfA \neq \{0\}$ for some $j \in \N$. Take $ba_jfh \in ba_jfA $ such that $ba_jfh \neq 0$. This means that $\bigcup_i I_{ba_if}\neq 0$. Also, for all $i\in \mathbb{N}$ and for all $h\in A$ we have  $ab\theta (a_ifhb)=0$ because $ab \in V$ and $a_ifhb\in a_ifA$. Assume that $\theta (a_ifhb):=tb$, then $\theta (ba_ifh)=bt$ because $R$ is reduced. We conclude that $ab\theta (ba_ifh)=0$, because $R$ is reduced  - since in every reduced ring $wyz=0$ implies $wzy=0$. 

Take $w=a$ and $y=bt$ and $z=b$. 
%Let $bt:=\theta (ba_ifh)$. 
So $abbt=ab^2t=0$. Again by using that $R$ is reduced, we get $abt=0$. So $a \theta (ba_i fh)=0$ for all $i \in \N$ and $h\in A$; and we conclude that $a\in \Ann (\bigcup_{i\in \mathbb{N}} I_{ba_if})$. 

Let $z\in V$. So $z\theta(a_ifl)=0$ for all $i>0$ and $l\in A$. In particular, $z\theta(a_ifhb)=0$ for all $h\in A$. Considering that $R$ is reduced, we get $z\theta(ba_ifh)=0$. This implies that $z\in \Ann(\bigcup_{i\in \mathbb{N}} I_{ba_if})$, which means that 
%The inclusion $ba_ifA\subseteq a_ifA$ for each $i$ gives us $I_{ba_if}\subseteq I_{a_if}$. Therefore,
\begin{align}\label{shar}
V\subseteq \Ann(\bigcup_{i\in \mathbb{N}} I_{ba_if}).
\end{align}
%Since $b\notin V$, we can see that $\Ann(I_{ba_if})\neq \{0\}$.
Also, since  $ba_ifA \neq 0$, for some $i \in \N$ and $ba_if^i=ba_jf^j \; \forall \; i,j\in \mathbb{N}$, we see that $\Ann(\bigcup_{i\in \mathbb{N}} I_{ba_if})\in T$. The equation
(\ref{shar}) and the maximal property of $V$ in $T$ yields that
$V=\Ann(\bigcup_{i\in \mathbb{N}} I_{ba_if})$. This implies that
$a\in V$ and that $V$ is a two sided completely prime ideal.

\textbf{Step 2.}
% Let $\omega_s:=\omega(s)$.
\textit{We show that $V$ is $\alpha$-invariant.} 

Let $r\in \alpha^{-1}(V)$. Then there exists $v\in V$ such that
$\alpha^{-1}(v)=r$ and we have $\alpha(r)=v$. By definition
of $V$, we have $\alpha(r)u=0$ for all $u\in \bigcup_{i\in
\mathbb{N}}(I_{a_if})$. Then, from Lemma~\ref{jav}, we have  $ru=0$ and it shows that
$r\in V$ and $\alpha^{-1} (V)\subseteq V$.

To show the reverse inclusion, take $v\in V$. Then for all $u\in \bigcup_{i\in \mathbb{N}}(I_{a_if})$ we have $vu=0$. Again, by Lemma~\ref{jav}, we have $\alpha(v)u=0$ . It shows that $\alpha(v)\in V$. Since $\alpha$ is surjective and rigid, hence an
automorphism, we see that $v \in \alpha^{-1}(V)$ and therefore, $V\subseteq \alpha^{-1}(V)$. From this we conclude that $V$ is an $\alpha$-invariant.

% Let $r\in V$. Then $rI_{a_i}=0$ by Lemma \ref{kav}. So $V\subseteq \omega_s(V)$ for each $s\in S$. Now let $r\in \omega_s(V)$, so $r\in \omega_s(t)$ such that $t\in V$ and $tI_{a_i}=0$. So $\omega_s(tI_{a_i})=0$ which means that
%$r\omega_s(I_{a_if})=\omega_s(t)\omega_s(I_{a_if})$.
%So $r\in \Ann(\omega_s(I_{a_if}))$. But we have $V\subseteq\omega_s(V)$ which means that
%$\Ann(I_{a_if})\subseteq \omega_s\left(\Ann(I_{a_if})\right)=\Ann\left(\omega_s(I_{a_if})\right)$.
% So $r\in \Ann(I_{a_if})=V$. Hence $\omega_s^{-1}(V)=V$ and $\omega_s$ is $V$-variant for each $s\in S$.

\textbf{Step 3.} \textit{ We show that $W:=(\frac{R}{V})[[x;\overline{\alpha}]]$ is a well-defined Archimedean domain.}

We already know that $V$ is a two-sided completely prime ideal, so the
factor ring $U:=\frac{R}{V}$ is a well-defined domain. Then it is enough to show that $U$ is an Archimedean domain because then by using
Theorem \ref{paykan} we can conclude that $U[[x;\overline{\alpha}]]$ is an Archimedean
domain, where $\overline{\alpha}: R/V \longrightarrow R/V$ defined by $\overline{\alpha}(a+V)=\alpha(a)+V$ for all $a\in R$ is an endomorphism of factor ring $R/V$.

%Assume that $U=\frac{R}{V}$ is not Archimedean. Then, there exists a nonunit element $\overline{y}\in U$ such that $\bigcap _{i\in \mathbb{N}}U\bar{y}^i\neq 0$. 

To show that $U$ is an Archimedean
domain, we take  $\overline{z}\in \bigcap _{i\in
\mathbb{N}}U\bar{y}^i $. Then $\overline{z}=\bar{z_1} \bar{y}=\bar{z_2} \bar{y}^2=\cdots$ for some $\overline{z_i}\in U$. It gives, \begin{align}\label{farzado}
z+l_0=z_1y+l_1=z_2y^2+l_2=\cdots.
\end{align}
for some $l_i \in V$. We choose $h\in A$ and $i\in
\mathbb{N}$. By multiplying  \ref{farzado}  with $\theta (a_ifh)$ and using that $l_i \in V$ for all $i$, we obtain
$\theta(a_ifh)z=\theta(a_ifh)z_1y=\cdots$. Then $\theta(a_ifh)z=0$, because $R$
is Archimedean and $y$ is not unit in $R$, otherwise $\overline{y}$ becomes a unit in $U$. As $h$ and $i$ are
chosen arbitrarily, we obtain that $z\in \Ann( I_{a_if})$ for all $i\in \mathbb{N}$ and therefore, $z\in V$. This shows, $\overline{z}=0$, as required. Therefore, $U$ is Archimedean.

\textbf{Step 4.} \textit{We claim that $\overline{f}$ is nonunit in $W=R/V[[x;\overline{\alpha}]]$.} 

Assume that there exists $g\in A$
such that $\overline{f}\overline{g}-\overline{1}=\overline{0}$. It means that all coefficients of $fg-1$ belong to $V$. In particular, $\theta (fg-1) \in V$ and we have  $\theta (fg-1)\theta (a_if)=0$ for all $i\in\mathbb{N}$. Assume that $f(x)=\sum f_lx^l$ and $g(x)=\sum g_l x^l$. We show that $f_0g_0$ is not a unit.
%Note that for any $h \in A$, $\pi (h) \geq 1_S$. Otherwise, if $\pi (h) < 1_S$, then $1_S>\pi(h)>\pi(h)^2>\cdots$
%is an infinite decreasing chain in $S$ which contradicts to the fact that $S$ is an Artinian positive monoid. Therefore,  $\pi(fg-1) \geq 1_S$.% Hence $\pi(f)\neq 1$ or $\pi(g)\neq 1$ which means1 $\overline{f}$ or $\overline{g}$ are not unit since $S$ is Artinian strictly totally ordered monoid, and the only unit of $S$ is $1$. So it is a contradiction.

\textbf{Case 1.}  Assume that $\pi(fg-1)>0$. 
%Let $k=fg-1$. Then $fg=1+k$ and
%$k_0=0$. 
%Observe that $\pi (f)\pi (g)=1_S$ because if $\pi (f)\pi (g) >1_S$, either $\pi (f)>1_S$ or $\pi(g) >1_S$. In any case, it would imply that $\theta(fg)\neq 1_R$ which is not true.
Then we have
%from (\ref{mult}) 
 $f_0g_0=\theta(fg)=1$. 
Thus, 
%$\theta (f)$ is a unit and
$f_0$ is a unit. %(\textcolor{red}{For the right case, we use the condition preserving nonunit here.}).
This implies that $f$ is a unit, which contradicts to our assumption that $f$ is non-unit. Therefore, we conclude that $\pi(fg-1)> 0$ is not possible.

\textbf{Case 2.} Now, let $\pi(fg-1)=0$. Note that $(fg)_0-1\in V$ and  $(fg)_0-1\neq 0,-1$, because otherwise it shows that  $V=A$, which is impossible. Since $(fg)_0\neq 0$, then $(fg)_0$ is a nonunit, otherwise $fg$ is unit, \cite[Proposition 3.2]{zim}, which is not possible. %(\textcolor{red}{Again, we used the condition preserving nonunits here.}). 
%It will lead to a contradiction that $f$ is unit.
 Therefore, $f_0g_0$ is a nonunit. 

We assumed that $\overline{f}$ is a unit, which implies that $\overline{f}^i$ is a unit for all $i \in \N$. For each $i \in \N$, let $\overline{b_i}$ be the inverse of $\overline{f}^i$. Then we have  $\overline{b_i}\overline{f}^i=\overline{b_j}\overline{f}^j=1$, for all $i,j \in \N$. 
It means that for all $i,j \in \N$ we have $\overline{b_i}\overline{f}^i-\overline{b_j}\overline{f}^j=0$ and  all the coefficients of $b_if^i-b_jf^j$ belong to $V$. It leads us to have 
$\theta(a_1f)(b_if^i-b_jf^j)=0$  - because $R$ is reduced. 

Therefore, we obtain
\begin{align}\label{equ}
\theta (a_1f)=\theta (a_1f)b_1f=\theta (a_1f)b_2f^2=\cdots.
\end{align}

We know from equation (\ref{sharto}) that  $\theta (a_1f)\neq 0$, which gives $\theta (a_1f)(b_if^i)_0\neq
0$.
% \textcolor{red}{- otherwise the right handsides of \eqref{equ} will be zero, while the very left hand side is nonzero. }.
 Let $b_i(x)=\sum b_{i,j}x^j$ and  $t_i:=\theta (a_1f)b_{i,0}$ and $r:=f_0$. 
%We compute $(b_if^i)_0$ as follows:
%\begin{align*}
One can see that $(b_if^i)_0
%=\sum_{xy=1}b_{i,x}\alpha^x(f_y^i)
=b_{i,0}f^i_0=b_{i,0}r^i$.
%\end{align*}
So $0\neq \theta (a_1f)(b_if^i)_0=t_ir^i$.
Thus, $t_{i+1}r^{i+1}=t_ir^i\neq 0$ for all $i$. Note that $r\neq 0$ is a non-unit because $f$ is non-unit. It sums up to the result that $R$  is non-Archimedean which is contradiction to our hypothesis. Therefore, $\overline{f}$ must be a non-unit.
% If $t_ir^{i}=0$ for some $i$, then $(t_ir)^{i}=0$ and since $R$ is reduced, $t_ir=0$.
%So $(a_1fh)(\pi(a_1fh))b_i(1)f(1)=0$. So $\overline{b_i}(1)\overline{f}(1)=\overline{0}$ which means that $\overline{f}(1)=\overline{0}$ or $\overline{b_i}(1)=\overline{0}$. If $\overline{b_i}(1)=\overline{0}$, then $\overline{b_i}$ is not a unit which is impossible (We know that $\overline{b_i}\overline{f}^i=\overline{b_{i-1}}\overline{f}^{i-1}$ by equation (\ref{shr}), so $\overline{b_i}\overline{f}^{i-1}=\overline{b_1}=\overline{1}$ which means that $\overline{b_i}$ is a unit). So $\overline{f}(1)=\overline{0}$ which is impossible according to our assumption,  $\overline{f}$ that is a unit.
%So for each $i$, $t_ir^{i+1}\neq 0$ which means that $t_ir^{i+1}=t_{i+1}r^{i+2}$.
%Since $R$ is Archimedean, $\bigcap_{i}t_ir^{i+1}=0$ which contradicts to the fact that $\overline{b_i}$ is a unit for each $i$.

\textbf{Step 5.} \textit{Finally, we are ready to show that $A$ is an Archimedean reduced ring.}

From equation (\ref{sharto}), we have
$\overline{a_1}\overline{f}=\overline{a_2}\overline{f^2}=\cdots$.
In Step 3 and 4, we showed that $W$ is Archimedean domain and $\overline{f}$ is a non-unit in $W$, so we must have $\overline{a_1}\overline{f}=\overline{0}$ and consequently, $\overline{a_1}=\overline{0}$ or
$\overline{f}=\overline{0}$. If $\overline{a_1}=\overline{0}$, then all coefficients of $a_1=\sum a_{1,j}x^j$ belong to $V$ and $a_{1,j}\theta(a_1f)=0$ for all $j\in \mathbb{N}$. 
%Then, from Lemma~\ref{jav}, we see that $\theta(a_1f)\alpha^t\left(a_{1,l}\right)=0$. 
%for
%$s,t\in S$. %Also, $R$ is $S-$rigid, so $\theta\left((a_1f)^2\right)=\left(\theta(a_1f)\right)^2$. 
Again, by using Lemma \ref{jav} we obtain
\begin{align*}
%\left((a_1f)(\pi(a_1f))\right)&\omega_{\pi(a_1f)}\left((a_1f)(\pi(a_1f))\right)=
\big( \theta (a_1f)\big)^2 =\bigg(\theta(a_1f)\bigg)\bigg(\theta(a_1f)\bigg)
%\nonumber\\
&=\theta(a_1f)\left(\sum_{x+y=\pi(a_1f)}a_{1,x}\alpha^x(f_y)\right)\nonumber\\
&=\sum_{x+y=\pi(a_1f)}\left(\theta(a_1f)a_{1,x}\alpha^x(f_y)\right)
%=\sum_{xy=\pi(a_1f)}\Bigg(\Big(\theta(a_1f)\omega_{\pi(a_1f)}\left(a_1(x)\right)\Big)\omega_{\pi(a_1f)}\left(\omega_x(f(y))\right)\Bigg)
%=\sum_{xy=\pi(a_1f)}0=0.
\end{align*}
Because $\theta (a_1f)a_{1,x}=0$ for all  $x\in S$, we see that
$\theta (a_1f)^2=0$. It gives that $\theta (a_1f)=0$ because $R$ is reduced. This gives us a contradiction to our assumption that $a_if \neq 0$ for all $i \in \N$. Therefore, $\overline{a_1} \neq 0$. 
%So $\left((a_1f)(\pi(a_1f))\right)\omega_{\pi(a_1f)}\left((a_1f)(\pi(a_1f))\right)=0$ which yields $\left((a_1f)(\pi(a_1f))\right)^2=0$ by Lemma \ref{jav}. Thus $(a_1f)(\pi(a_1f))=0$ which is a contradiction.
Then we must have $\overline{f}=\overline{0}$, that is, all coefficient of $f$ must belong to $V$ and $f_j\theta (a_1f)=0$ for all $j\in \mathbb{N}$. Then, from Lemma~\ref{jav}, we see that $\theta(a_1f)\alpha^l\left(f_j\right)=0$. This gives,  
\begin{align}
%\theta((a_1f)^2)=
\bigg(\theta(a_1f)\bigg)^2=\sum_{x+y=\pi(a_1f)}a_{1,x}\alpha^x(f_y)\theta(a_1f)=0
\end{align}
which again implies that $\theta((a_1f))^2=0$ and contradicts our assumption that $a_if \neq 0$ for all $i \in \N$. Therefore, our primary assumption $\Gamma \neq \emptyset$ is false. It shows that $A$ is an Archimedean ring, which completes the proof of the "only if" part.

Conversely,  assume that $\bigcap _{n=1}^{\infty}Ra^n \neq 0$ for some $a\in R$. Then $0\neq \bigcap _{n=1}^{\infty}Ra^n \subseteq \bigcap _{n=1}^{\infty} Aa^n$. But $A$ is a right Archemidean ring, so $a$ should be a unit in $A$. Thus, $fa=1$ for some $f\in A$. This means that $(fa)_0=1$. 
%Considering the fact that $S$ is Artinian positively strictly ordered monoid, 
So we get $f_0a=1$. Hence $a$ is a unit in $R$, which leads us to the conclusion that $R$ is a right Archemidean ring.

We show that $\alpha$ preserves nonunits. If $b\in R$ is a nonunit element and $\alpha(b)$ is a unit, then $c\alpha(b)=1$ for some $c\in R$. Assume that $f(x)b=1$, then $f_0b=1$ which is not possible. So $b$ is not a unit in $A$. Also, $c^nxb^{n}=c^n\alpha(b^n)x=x=c^mxb^m$. So $x\in \bigcap_{n} Ab^n$, while $b$ is not a unit in $A$. This contradiction shows that $\alpha(b)$ should not be a unit in $R$.
%, which means that $\alpha$ preserves nonunit elements of $R$.
\end{proof}

In the case $\alpha=id_R$ - but $R$ is still not necessarily commutative - we conclude the next corollary immediately.
 
\begin{corollary}\label{archcor}
Let $R$ be a right (resp. left) Archimedean  reduced ring which satisfies the ACC on annihilators.
%\begin{itemize}
%\item[(1)] 
Then the power series ring $R[[x]]$ is a right (resp. left)
Archimedean  reduced ring.
%\item[(2)] Let  $\alpha$ be a
%rigid automorphism of $R$. Then the skew power series ring $R[[x;\alpha]]$ is a right (resp.
%left) Archimedean reduced ring. 
%\end{itemize}
\end{corollary}

%\textcolor{red}{In fact we can leverage the proof of theorem \ref{rigidar} have the following theorem, which we will leave its proof to the reader.}
%
%\begin{theorem}\label{generalrigidar}
%Let  $(S,\le)$ be an Artinian, positive and \textcolor{red}{strictly} totally ordered monoid, $R$ be a $S$-rigid ring satisfying the ACC on annihilators. Furthermore, let
%$\omega: S \rightarrow \End(R)$
%be a monoid homomorphism and $\omega_s$ is surjective for all $s\in S$ then:
%\begin{enumerate}
%\item[\em{(1)}] $R$ is right Archimedean and $\omega_s$ preserves nonunits of $R$ for any $s\in S$ if and only if the reduced ring $R[[S,\omega]]$ is right Archimedean .
%\item[\em{(2)}] $R$ is left Archimedean if and only if the reduced ring $R[[S,\omega]]$ is left Archimedean.
%\end{enumerate} 
% %\textcolor{red}{We can write this theorem like theorem \ref{domar}. Because for the right case, we need to preserve nonunits of $R$ for any $s\in S$.}
%\end{theorem}
%%\begin{theorem}\textcolor{green}{(Converse of last Theorem)}
%%Let $S$ be an Artinian positively strictly ordered monoid and $R[[S,\omega]]$ is a right (left) Archemidean $S$-rigid ring. Then $R$ is right (left) Archemidean ring.
%%\end{theorem}
%%
%%\begin{proof}
%%We prove only for the right case and the similar proof hold for the left case.
%%\end{proof}

%\begin{proof}
%Both (1) and (2)  follows from Theorem~\ref{rigidar}.
%\end{proof}

%\section{Conditions required to transfer the property of being Archimedean and satisfying ACCP from $R$ to $R[[S,\omega]]$}\label{eexamples1}

In Theorem~\ref{rigidar}, we proved that under some conditions on the ring $R$, it is possible to transfer the property of being Archimedean from $R$ to $R[[x;\alpha]]$. 
%\textcolor{red}{- or a more general case in theorem \ref{generalrigidar}}. 
It is natural to ask that what are the minimum conditions required on $R$ to allow such a transfer from $R$ to $R[[S,\omega]]$. With the help of the next example, we justify that the condition rigid required in Theorem~\ref{rigidar} cannot be changed to reduced.

\begin{example}\label{reducekhubu}
%This example shows that in Theorem~\ref{rigidar}, we cannot change the condition on $R$ from being $S$-rigid ring to a reduced ring.
% Let $R:=\mathbb{F}[[x,y,y^{-1}]]$, where $\mathbb{F}$ is a field.
%
%Consider the skew generalized power series ring $R[[S,\omega]]$.  we get $R=\mathbb{F}$ where 
Let $\mathbb{F}$ be any field and put $S:=(\mathbb{N}\cup \{0\})\times (\mathbb{Z})$
with usual addition and usual order.
 %and let $\omega :S \rightarrow \End(\mathbb{F})$ be the trivial map, then
 Consider $\mathbb{F}[[S]]= \mathbb{F}[[(\mathbb{N}\cup \{0\})\times (\mathbb{Z})]]:=T$. If $\mathbb{F}[[\mathbb{Z}]]:=R$, then one can see that $T=R[[\mathbb{N}\cup \{0\}]]:=R[[S]]$. Also, we can see that $R$ is isomorphic to the topological closure of $\mathbb{F}((x))$ where the metric is induced from the p-adic valuation $\deg(\cdot)$ and $x:=e_{(1,0)}\in T$ - we use the standard notation $e_s$ for $s\in S$ in the  skew generalized  power series rings. 

Define $x:=e_{(1,0)}\in T$ and $y:=e_{(0,1)}\in T$.
%where $e_s$ is defined in \ref{e_s}. 
One can see that $T$ is an skew generalized power series rings of $R$ where $S$ is a positive stirctly totally ordered monoid. 
Then, $T$ is a Noetherian integral domain (i.e. a reduced ring). Also, $T$ satisfies both \emph{ACCP} and \emph{ACC} on annihilators.
Take $\alpha :T\rightarrow T$ defined by 
\begin{align}
\alpha (1)=1, \quad \alpha (x)=y^{-1}, \quad \alpha (y)=y. 
\end{align} 
Note that $\alpha$ is not a monomorphism, and therefore, not rigid because $\alpha (xy^2-y)=0$ and $xy^2-y \neq 0$.  So $T$ is not an $\alpha$-rigid ring.
 
Let $U:=T[[z;\alpha]]$ and $f_k(z)=x\sum _{n=k}^{\infty}y^n z \in T$. Then 
\begin{align}
f_{k+1}(z).x=\big (\sum _{n=k+1}^{\infty} xy^nz \big )x=\sum _{n=k+1}^{\infty}x\alpha (x)y^n z=\sum _{n=k+1}^{\infty}xy^{n-1}z=\sum _{n=k}^{\infty}xy^nz=f_k(z).
\end{align}
It shows that $ f_k U \subseteq  f_{k+1} U$ and we obtain an ascending chain of principal ideals in $T$. In order to show that $f_k U \subsetneq  f_{k+1} U$, it is enough to show that $x$ is not unit in $U$. Suppose that $h(z)x=1$ for some $h(z) \in U$. Let $h(z)=\sum _{i=0}^{\infty}g_i(x,y,y^{-1})z^i$, then 
$h(z)x=g_0(x,y,y^{-1})x+(\sum _{i=1}^{\infty}g_i(x,y,y^{-1})z^i)x=1$.
It implies, $g_0(x,y,y^{-1})x=1$ showing that $x$ is a unit in $T$, which is not true. Therefore, we obtain a non-stablizing chain in $U$ given by $f_1 T \subsetneq  f_{2} U\subsetneq \ldots$. Hence, $U$ does not satisfy \emph{ACCPR}.
\end{example}

In \cite{lantz}, the authors give another example which also shows that the $\alpha$-rigid condition in Theorem~\ref{rigidar} is not superfluous. A ring is $\alpha$-compatible if $a\alpha(b)=0$ results in $ab=0$ - see more information about $\alpha$-compatible rings in \cite{moussavi}. It is known that the set of $\alpha$-rigid rings are the intersection of the set of $\alpha$-compatible rings and reduced rings. We proved that Archimedean property can be shifted from the $\alpha$-rigid rings and propose a counterexample for the reduced rings. The remaining natural question which seeks an answer is:
\begin{question}
Let  $R$ be an $\alpha$-compatible ring satisfying the \emph{ACC} on annihilators. Furthermore, let  $\alpha$ be a surjective endomorphism in $R$. Then $R$ is right (resp. left) Archimedean if and only if the reduced ring $R[[x;\alpha]]$ is right (resp. left) Archimedean.
\end{question}

\end{document}